\documentclass[a4paper]{article}
\usepackage{amsthm,amssymb,amsmath,enumerate}
\usepackage{algorithm}

\usepackage{tikz}
\usetikzlibrary{backgrounds}
\usetikzlibrary{decorations}
\usetikzlibrary{decorations.pathmorphing}

\tikzstyle{hvertex}=[thick,circle,inner sep=0.cm, minimum size=2mm, fill=white, draw=black]
\tikzstyle{hedge}=[very thick]
\tikzstyle{rededge}=[very thick,red]
\tikzstyle{point}=[draw,circle,inner sep=0.cm, minimum size=1mm, fill=black]
\tikzstyle{pointer}=[thick,->,shorten >=2pt,color=hellgrau]
\tikzstyle{pathedge}=[hedge,decorate, decoration={random steps,segment length=3pt,amplitude=1pt}]

\colorlet{auchblau}{blue!60!white}
\colorlet{hellblau}{blue!20!white}
\colorlet{hellrot}{red!40!white}
\colorlet{hellgrau}{black!30!white}
\colorlet{dunkelgrau}{black!60!white}

\colorlet{grau}{black!50!white}

\newtheorem{definition}{Definition}
\newtheorem{proposition}[definition]{Proposition}
\newtheorem{theorem}[definition]{Theorem}

\newtheorem{lemma}[definition]{Lemma}

\newtheorem{problem}[definition]{Problem}

\newcommand{\bigO}{O}

\newcommand{\cP}{\mathcal{P}}
\newcommand{\cQ}{\mathcal{Q}}
\newcommand{\cS}{\mathcal{S}}

\newcommand{\cR}{\mathcal{R}}

\newcommand{\comment}[1]{}
\newcommand{\N}{\mathbb N}
\newcommand{\R}{\mathbb R}

\newcommand{\emtext}[1]{\text{\em #1}}

\newcommand{\sm}{\setminus}

\newcommand{\thofun}{h_{\ref{thoprop}}}
\newcommand{\bfun}{h_{\ref{linkagewallminor}}}
\newcommand{\rsfun}{h_{\ref{robseywalltangle}}}
\newcommand{\fptfun}{h_{\ref{oddflatwallthm}}}
\newcommand{\pfun}{h_{\ref{bipblockcase}}}
\newcommand{\plfun}{h_{\ref{pathlinkagetotop}}}

\title{Packing $A$-Paths of Length Zero Modulo Four}
\author{Henning Bruhn and Arthur Ulmer\thanks{supported by DFG, grant no.\  BR 5449/1-1}}
\date{}

\begin{document}
\maketitle

\begin{abstract}
We show that $A$-paths of length $0$ modulo $4$
have the Erd\H{o}s-P\'osa property. 
Modulus~$m=4$ is the only composite number for which  $A$-paths of length~$0$ modulo~$m$ have the property. 
\end{abstract}

\section{Introduction}
Cycles obey a packing-covering duality, as Erd\H os and P\'osa proved 
in their now classic 1965 paper~\cite{EP65}: every graph contains $k$
disjoint cycles, or a vertex set of size $\bigO(k\log k)$
that meets every cycle. Gallai's~\cite{Gal61} earlier result about $A$-paths
can be phrased in similar terms: for every vertex set $A$, every graph 
contains $k$ disjoint $A$-paths or a vertex set of size at most $2k-2$
that meets every $A$-path. (An \emph{$A$-path} is a path 
that starts in $A$, ends in a different vertex of $A$, and has no intermediate vertex in $A$.)

More succinctly, we might say that cycles, as well as $A$-paths have 
the \emph{Erd\H{o}s-P\'osa property}. Here,  
 a class of graphs\footnote{Or graphs with some extra structure, such 
as $A$-paths.} has the Erd\H{o}s-P\'osa property if 
there is a function $f: \N \to \R_+$ such that in every graph there are either $k$ 
disjoint subgraphs 
belonging to the class  or a vertex set $X$ of size $|X| \leq f(k)$ 
that intersects all subgraphs that lie in the class. 

A fairly general class  with the Erd\H os-P\'osa property is 
the class of graphs that have a fixed planar graph as minor~\cite{RS86}. Also special
types of $A$-paths have it. Indeed, 
Wollan~\cite{Wol10} proved that for any $m$, the class of $A$-paths of length $\neq 0$ 
modulo $m$ has the property.\footnote{Actually, Wollan's result is more comprehensive; 
see Section~\ref{bipsec}.} This includes, in particular, odd $A$-paths. 

So,
what about $A$-paths of length \emph{equal} to $0$ modulo $m$? Bruhn, Heinlein and
Joos~\cite{BHJ18b} found that even $A$-paths, i.e.\ when $m=2$,
 have the Erd\H{o}s-P\'osa property, and constructed
counterexamples that show that the property is lost for composite numbers $m > 4$. 
In particular, the only composite number $m$ for which it is open whether the 
property holds or not is $m=4$.

When it comes to the Erd\H os-P\'osa property,
there is seemingly a common phenomenon:
 counterexamples are easy to find, 
and normally if none is found then the Erd\H{o}s-P\'osa property holds.
This is also the case here. We prove:
 
\begin{theorem}\label{zeroepp}
$A$-paths of length $0$ modulo $4$ have the Erd\H{o}s-P\'osa property.
\end{theorem}

Our proof relies on two components. First, we use that if the 
theorem fails in a graph, then that graph admits a large tangle that 
always points to where the desired $A$-paths are found in the graph.
A similar approach has also been used by others. 
Second, we use a  powerful structural result by Huynh, Joos and Wollan~\cite{HJW16}
that gives insight on where, with respect to the tangle, the odd cycles 
of the graph are located.

\medskip

There is a rather rich literature on the Erd\H os-P\'osa property. 
In particular, a number of classes of cycles and paths with additional 
restrictions on the lengths are known to have the 
property.
These include: $A$-paths with a fixed minimum length~\cite{BHJ18b}, 
cycles of length $0$ modulo $m$~\cite{Tho88}, and also cycles of length 
$\neq 0$ modulo $m$ for odd $m$~\cite{Wol11}. 

Kriesell~\cite{Kri05} proved a directed analogue of Gallai's $A$-path theorem, 
which means that (directed) $A$-paths in digraphs have the Erd\H os-P\'osa property.
The same holds for directed cycles~\cite{RRST96}, as well as directed cycles
of a fixed minimum length~\cite{KK14}.
Many more results may be found in the survey of Raymond and Thilikos~\cite{RT16},
and in~\cite{BHJ18b}.

\section{Walls, tangles and minors}
We need a number of tools from the graph minors project of Robertson and Seymour
that we define in this section. For general graph-theoretic notation
we refer to Diestel~\cite{diestelBook17}.

Denote by $[n]$ the set $\{1,..,n\}$. We define an \emph{$n \times m$ grid} as the graph on the vertex set $[n] \times [m]$ with edges $(i,j)(k,l)$ if and only if $|i-k|+|j-l|=1$. An \emph{elementary $n$-wall} is a subgraph of an $(n+1) \times (2n+2)$ grid where we delete all edges $(2i-1,2j)(2i,2j)$ for $i \in [\lceil \frac{n}{2} \rceil]$ and $j \in [n+1]$ and all edges $(2i,2j-1)(2i+1,2j-1)$ for all $i \in [\lfloor \frac{n-1}{2} \rfloor]$ and $j \in [n+1]$ and afterwards deleting all vertices of degree one. Figure~\ref{fig:6-wall} shows a drawing of the wall.

 There is a unique collection of $n+1$ disjoint paths from vertices $(1,i)$ to vertices $(n+1,j)$ (where $i,j \in [2n+1]$ if $n$ is odd and $j\in [2n+2] \sm \{1\}$ if $n$ is even); these are the \emph{vertical paths}. Let $P_1$ be the vertical path containing $(1,1)$ and $P_2$ the one containing $(1,2n+1)$. There is again a unique collection of $n+1$ disjoint paths from $P_1$ to $P_2$; these are the \emph{horizontal paths}.
We can order the horizontal paths from top to bottom and the vertical paths from left to right. 
We say the \emph{first/second/\dots/last horizontal/vertical path} for 
the path that is the first/second/\dots/last in this order. The first horizontal path is the \emph{top row}.

The \emph{nails} of an elementary wall are the interior vertices of the top row of degree $2$ in the wall,
and the \emph{outer cycle} is the union of the first and last horizontal path and the first and last vertical path. 
A \emph{brick} is any cycle of length six in an elementary wall.

\begin{figure}[ht]
\centering
\begin{tikzpicture}
\tikzstyle{hvertex}=[thick,circle,inner sep=0.cm, minimum size=1.6mm, fill=white, draw=black]
\tikzstyle{marked}=[line width=3pt,color=dunkelgrau]
\tikzstyle{point}=[thin,->,shorten >=2pt,color=dunkelgrau]

\def\wallheight{8}
\def\brickheight{0.5}

\pgfmathtruncatemacro{\lastrow}{\wallheight}
\pgfmathtruncatemacro{\penultimaterow}{\wallheight-1}
\pgfmathtruncatemacro{\lastrowshift}{mod(\wallheight,2)}
\pgfmathtruncatemacro{\lastx}{2*\wallheight+1}

\draw[hedge] (\brickheight,0) -- (2*\wallheight*\brickheight+\brickheight,0);
\foreach \i in {1,...,\penultimaterow}{
  \draw[hedge] (0,\i*\brickheight) -- (2*\wallheight*\brickheight+\brickheight,\i*\brickheight);
}
\draw[hedge] (\lastrowshift*\brickheight,\lastrow*\brickheight) to ++(2*\wallheight*\brickheight,0);

\foreach \j in {0,2,...,\penultimaterow}{
  \foreach \i in {0,...,\wallheight}{
    \draw[hedge] (2*\i*\brickheight+\brickheight,\j*\brickheight) to ++(0,\brickheight);
  }
}
\foreach \j in {1,3,...,\penultimaterow}{
  \foreach \i in {0,...,\wallheight}{
    \draw[hedge] (2*\i*\brickheight,\j*\brickheight) to ++(0,\brickheight);
  }
}

\def\colind{5}
\foreach \j in {2,4,6}{
  \draw[marked] (\colind*\brickheight,\j*\brickheight-2*\brickheight) -- ++ (0,\brickheight) -- ++(-\brickheight,0) -- ++(0,\brickheight) -- ++(\brickheight,0);
}
\draw[marked] (\colind*\brickheight,6*\brickheight) -- ++ (0,\brickheight) -- ++(-\brickheight,0) -- ++(0,\brickheight);

\def\rowind{4}
\foreach \i in {1,...,\lastx}{
  \draw[marked] (\i*\brickheight-\brickheight,\rowind*\brickheight) -- ++(\brickheight,0);
}

\draw[marked] (2*\wallheight*\brickheight,1*\brickheight) -- ++(0,\brickheight) coordinate[midway] (brx)
-- ++(-2*\brickheight,0)
-- ++(0,-\brickheight) -- ++(2*\brickheight,0);

\foreach \i in {1,...,\lastx}{
  \node[hvertex] (w\i w0) at (\i*\brickheight,0){};
}
\foreach \j in {1,...,\penultimaterow}{
  \foreach \i in {0,...,\lastx}{
    \node[hvertex] (w\i w\j) at (\i*\brickheight,\j*\brickheight){};
  }
}
\foreach \i in {1,...,\lastx}{
  \node[hvertex] (w\i w\lastrow) at (\i*\brickheight+\lastrowshift*\brickheight-\brickheight,\lastrow*\brickheight){};
}

\foreach \i in {2,4,...,\lastx}{
  \node[hvertex,fill=white] (w\i w\lastrow) at (\i*\brickheight+\lastrowshift*\brickheight-\brickheight,\lastrow*\brickheight){};
}

\node[anchor=mid] (tr) at (\lastx*\brickheight+0.5,\wallheight*\brickheight+0.8){top row};
\draw[point,out=270,in=0] (tr) to (w\lastx w\wallheight);

\node[anchor=mid] (nails) at (\lastx*\brickheight-1,\wallheight*\brickheight+0.8){nails};
\draw[point,out=180,in=90] (nails) to (w10w\wallheight);
\draw[point,out=190,in=90] (nails) to (w12w\wallheight);
\draw[point,out=200,in=90] (nails) to (w14w\wallheight);

\node[anchor=mid] (vp) at (0,\wallheight*\brickheight+0.8){vertical path};
\draw[point,out=0,in=90] (vp) to (w\colind w\wallheight);

\node[align=center] (hp) at (\lastx*\brickheight+1.2,\rowind*\brickheight+0.8){horizontal\\ path};
\draw[point,out=270,in=0] (hp) to (w\lastx w\rowind);

\node[align=center] (br) at (\lastx*\brickheight+1.2,1*\brickheight+0.8){brick};
\draw[point,out=270,in=0] (br) to (brx);

\end{tikzpicture}
\caption{An elementary $8$-wall}\label{fig:6-wall}
\end{figure}
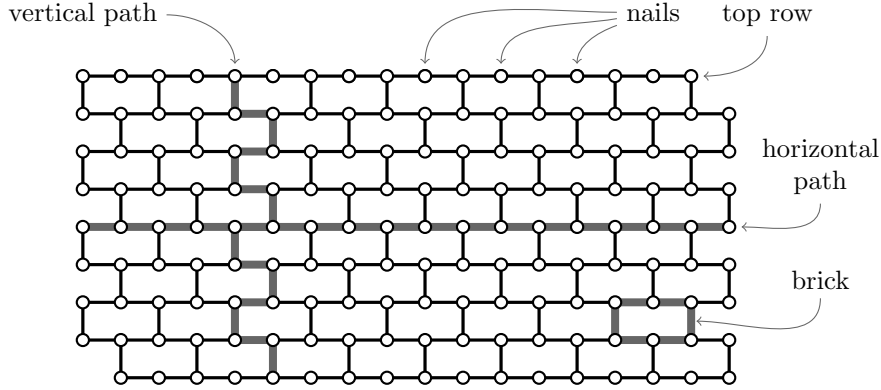

A \emph{wall of size~$n$}, or an \emph{$n$-wall}, is a subdivision of an elementary $n$-wall. 
All definitions above can be extended to subdivisions of walls by replacing each vertex of the elementary wall with its branch vertex in the subdivision. However, since nails are vertices of degree~$2$, 
there are usually multiple ways to choose the branch vertex of a nail. 

Let $W'$ be a wall that is contained in a wall $W$.  We say $W'$ is a \emph{subwall} of $W$ if each horizontal path of $W'$ is a subpath of a horizontal path of $W$ and the same is true 
for the vertical paths. Additionally, we require that whenever $W'$ contains a subpath of the $i$th and of the $j$th horizontal path of $W$ 
then it also contains a subpath of the $\ell$th horizontal path for all $i<\ell<j$, and 
 the same holds for vertical paths.

A subwall $W'$ is \emph{$k$-contained} in $W$ if it is disjoint of the first $k$ and the last $k$ horizontal and vertical paths of $W$. For a subwall $W'$ that is at least $1$-contained in a wall $W$ there is a natural choice of nails:
those vertices in the top row of $W'$ that are branch vertices in $W$. 
Whenever we have an at least $1$-contained subwall of $W$ we will always assume the nails to be chosen in this way.

A \emph{separation} in a graph $G$ is a pair $(C,D)$ such that $C$ and $D$ are edge-disjoint subgraphs of $G$ and $C \cup D = G$. We define its \emph{order} as the cardinality of $V(C \cap D)$. A \emph{tangle of order $n$} is a set $\mathcal{T}$ of separations (in $G$) of order $\leq n-1$ with the following properties:
\begin{enumerate}[(T1)]
\item\label{tangle1} For every separation $(C, D)$ of order $\leq n-1$ either $(C, D) \in \mathcal{T}$ or \mbox{$(D, C) \in \mathcal{T}$} but not both
\item\label{tangle2} $V(C) \neq V(G)$ for all $(C, D) \in \mathcal{T}$
\item\label{tangle3} $C_1 \cup C_2 \cup C_3 \neq G$ for all $(C_1, D_1), (C_2, D_2), (C_3, D_3) \in \mathcal{T}$
\end{enumerate}

In the following we will not require the two subgraphs $C$ and $D$ to be edge-disjoint to form a separation. This is because the way our tangles are defined the big side only depends on the position of some vertices (not edges).

When a graph $G$ has a graph $H$ as a minor, then it contains an \emph{$H$-model}, that is, there 
is a mapping $\pi$ from $V(H)\cup E(H)$ into $G$ such that
\begin{itemize}
\item $\pi(v)$ is a tree in $G$ (the \emph{branch set} of $v$) 
for each $v\in V(H)$, and two such trees $\pi(v)$, $\pi(u)$
for distinct $u,v\in V(H)$ are disjoint; and
\item $\pi(uv)$ is an edge of $G$ between $\pi(u)$ and $\pi(v)$ for each $uv\in E(H)$.
\end{itemize}

A $K_t$-model $\pi$ is an \emph{odd $K_t$-model} if
the unique cycle in $\pi(u)\cup\pi(v)\cup\pi(w)\cup\{\pi(uv),\pi(vw),\pi(wu)\}$
is odd for all distinct $u,v,w\in V(K_t)$.

If $G$ contains an $H$-model $\pi$ then every tangle $\mathcal T$ of $H$ of order $n$ induces 
a tangle $\mathcal T_\pi$ in $G$ of the same order. Indeed, let $(C, D)$ be a separation in $G$ of order $\leq n-1$,
and let $C_H, D_H \subseteq H$ be the induced subgraphs of $H$ obtained by 
putting all vertices of $H$ whose branch sets under $\pi$ 
are intersecting $C$ into  $C_H$ and all vertices whose branch sets are intersecting $D$ into $D_H$. 
Then, $C_H \cup D_H =H$, and $(C_H,D_H)$ is a separation in $H$ of order at most $n-1$. 
Therefore, either $(C_H, D_H) \in \mathcal{T}$ or $(D_H, C_H) \in \mathcal{T}$, 
and we then put $(C, D)$ resp.\ $(D, C)$ into $\mathcal{T}_{\pi}$. This defines a tangle~\cite{RS91}.

We describe two elementary tangles that we will need. 
First, we define a tangle of order $n = \lceil \frac{2t}{3} \rceil$ in a complete graph $K_t$. For any separation $(C,D)$ of order at most $n-1$ in a complete graph there is one side, $D$ say, which contains all of $V(K_t)$. Putting all such $(C,D)$ into a set $\mathcal{T}$ defines a 
tangle~\cite{RS91}. 

Next, consider a wall $W$ of size $n$ in a graph $G$. 
In any separation $(C,D)$ of $G$ of order $\leq n-1$ there has to be exactly one side, $C$ or $D$, 
where a complete horizontal path of the wall lies. Then, there is a tangle $\mathcal T_W$
of order $n$ defined as follows: whenever $(C,D)$ is a separation of $G$ of order less than $n$, 
such that $D$ contains a complete horizontal path of $W$, put $(C,D)$ into $\mathcal T_W$.
Again, this defines a tangle~\cite{RS91}.

Let $\mathcal{T}$ be a tangle of order $n$. Let $m \leq n$ be some positive integer,
 and let $\mathcal{T}_0$ be the subset of $\mathcal{T}$ that consists of all separations in $\mathcal T$
 of order at most $m-1$. 
Then $\mathcal{T}_0$ is a tangle of order $m$, the \emph{truncation} of $\mathcal{T}$.
We need an elementary lemma about the truncations of a wall tangle. For a proof see for instance~\cite{HJW16}.
\begin{lemma}\label{subwalltangle}
If $W_0$ is a subwall of a wall $W$, then $\mathcal{T}_{W_0}$ is a truncation of $\mathcal{T}_W$.
\end{lemma}

In any tangle $\mathcal{T}$ of order $n$ and set $X \subseteq V(G)$ of size at most 
$n-2$ there is a unique block $U$ of $G - X$, the \emph{$\mathcal{T}$-large block} of $G - X$,  
such that $X \cup V(U)$ never lies in $C$ if $(C,D) \in \mathcal{T}$; see for instance \cite{HJW16}.

{
Over the course of this article, we will need to use a number of functions that 
usually force some structure in a graph. We write these functions as $h_i$, where
$i$ is the number of the theorem the function occurs in. The first example is 
function $\rsfun$ in the next result.
}

\begin{theorem}[Robertson, Seymour and Thomas \cite{quicklyExcluding}]\label{robseywalltangle} 
For every positive integer $t$ there is an integer $\rsfun(t)$ such that in any graph with a tangle $\mathcal{T}$ of order $\rsfun(t)$  there is a $t$-wall $W$ such that $\mathcal{T}_W$ is a truncation of $\mathcal{T}$.
\end{theorem}

A \emph{linkage of a wall $W$} with nails $N$ is a set of disjoint $N$-paths contained in $G-(W-N)$. 
The top row of $W$ induces an ordering $\leq$ on $N$ (in fact, it induces two --- we pick one). 
For a linkage path $P$, we write $r_P$ and $\ell_P$ for the endvertices of $P$ if $\ell_P<r_P$.
The linkage $\mathcal L$ is \emph{in series} if $r_P<\ell_Q$  for all distinct $P,Q\in\mathcal L$ with $\ell_P<\ell_Q$;
it is  \emph{crossing} if $\ell_P<\ell_Q<r_P<r_Q$ 
for all distinct $P,Q\in\mathcal L$ with $\ell_P<\ell_Q$; 
and $\mathcal L$ is \emph{nested} if $\ell_P<\ell_Q<r_Q<r_P$
for all distinct $P,Q\in\mathcal L$ with $\ell_P<\ell_Q$; see Figure~\ref{linkagefig}. 
The linkage is \emph{pure} if it is either crossing, nested or in series.

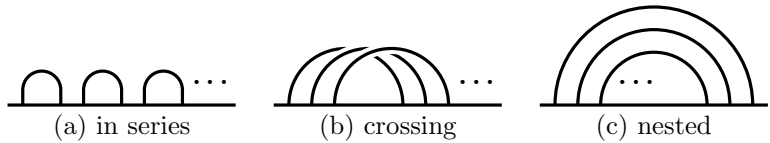
\begin{figure}[ht]
\centering
\begin{tikzpicture}

\tikzstyle{oben}=[line width=1.5pt,double distance=1.2pt,draw=white,double=black]

\begin{scope}
\def\radius{0.25}
\def\step{0.3}
\draw[hedge] (0,0) -- (3,0);
\draw[hedge] (0.2,0) -- ++(0,0.2) arc (180:0:\radius) -- ++(0,-0.2);
\draw[hedge] (0.2+2*\radius+\step,0) -- ++(0,0.2) arc (180:0:\radius) -- ++(0,-0.2);
\draw[hedge] (0.2+4*\radius+2*\step,0) -- ++(0,0.2) arc (180:0:\radius) -- ++(0,-0.2);
\node at (2.7,0.3) {{\bf \dots}};
\node at (1.5,-0.3) {(a) in series};
\end{scope}

\begin{scope}[shift={(3.5,0)}]
\draw[oben] (0.2,0) arc (180:0:0.75);
\draw[oben] (0.2+0.3,0) arc (180:0:0.75);
\draw[oben] (0.2+2*0.3,0) arc (180:0:0.75);
\draw[hedge] (0,0) -- (3,0);
\node at (2.7,0.3) {{\bf \dots}};
\node at (1.5,-0.3) {(b) crossing};
\end{scope}

\begin{scope}[shift={(7,0)}]
\draw[hedge] (0,0) -- (3,0);
\draw[hedge] (0.2,0) arc (180:0:1.3);
\draw[hedge] (0.2+0.3,0) arc (180:0:1.0);
\draw[hedge] (0.2+2*0.3,0) arc (180:0:0.7);
\node at (1.3,0.3) {{\bf \dots}};
\node at (1.5,-0.3) {(c) nested};
\end{scope}

\end{tikzpicture}
\caption{The three types of pure linkages}\label{linkagefig}
\end{figure}

An \emph{odd linkage} of a bipartite wall $W$ is a linkage of $W$ such that for every 
path $P$ in the linkage, every cycle in $P \cup W$ that passes through $P$ is odd, 
or equivalently, $P \cup W$ is not bipartite.

The principal tool for our proof is a powerful structural result by Huynh, Joos and Wollan~\cite{HJW16}.
We present here a simplified version that is adapted to our needs. 
The original formulation covers graphs in which edges are endowed with two 
\emph{directed} group-labellings.\footnote{Our version arises by working in the group $\mathbb Z_2$
and labelling every edge with~$1$.} Moreover, the conclusion of the theorem 
is stronger: in the original version the obtained wall $W_0$ is \emph{flat}, that is, 
embedded in an essentially planar part of the graph. We never use the flatness of $W_0$
and therefore omit it from the statement.

\begin{theorem}[Huynh, Joos and Wollan~\cite{HJW16}]\label{oddflatwallthm}
For every $t\in\mathbb N$, there exist an integer $\fptfun(t)$ with the
following property. 
Let $G$ be a graph that contains a $\fptfun(t)$-wall $W$. Then one of the following
statements holds.
\begin{enumerate}[\rm (a)]
\item There is an odd $K_t$-model $\pi$ in $G$ such 
that $\mathcal T_\pi$ is a truncation of $\mathcal T_W$.
\item There  is a  $100t$-wall
$W_0$ in $G$ such that $\mathcal T_{W_0}$ is a truncation of $\mathcal T_W$ such that
\begin{enumerate}[\rm (b.i)]
\item every brick of $W_0$ is an odd cycle; or
\item $W_0$ is bipartite, and there is a  pure odd linkage
$\mathcal P$ of $W_0$ of size $t$.
\end{enumerate}
\item There is a set $Z$ of vertices of $G$ such 
that $|Z|\leq \fptfun(t)$ and the $\mathcal T_W$-large block of $G-Z$ is bipartite.
\end{enumerate}
\end{theorem}

{For the proof of the main result, we will deal with the different outcomes of Theorem~\ref{oddflatwallthm}
separately.}

\section{Bipartite case}\label{bipsec}

{
For any graph $G$ and set $A\subseteq V(G)$,
we say that $X\subseteq V(G)$ is a \emph{hitting set}
if $G-X$ does not contain any $A$-path of length~$0$~modulo~$4$.}
The hitting {sets} we will construct later have a very large size. It is possible, though, 
to get a much smaller hitting set if we assume the graph $G$ to be bipartite.
In fact, it suffices that $G-A$ is bipartite. This is what we do in this section.

At the same time, we here lay some ground work for the 
main theorem, by  dealing with the last case of 
Theorem~\ref{oddflatwallthm}. For the proof of the bipartite case we 
need the theorem by Wollan on $A$-paths that we already mentioned in the introduction.

{An \emph{undirected $\Gamma$-labelling}} in a graph $G$ is an assignment $\gamma$ that assigns to every edge of $G$ a value from an {abelian} group $\Gamma$. 
{If $P$ is a path in $G$ then its \emph{weight} $\gamma(P)$ is 
defined as $\gamma(P)=\sum_{e\in E(P)}\gamma(e)$.
We say that $P$ is a \emph{zero} path \emph{with respect to $\gamma$} if $\gamma(P)=0$, and 
it is a \emph{non-zero} path if $\gamma(P)\neq 0$.
If we do not explicitly specify a group labelling, we assume implicitly a labelling of~$1$
on every edge in the group~$\mathbb Z_4$, which means that a path is a zero path
if and only if its length is~$0$ modulo~$4$.}

\begin{theorem}[Wollan~\cite{Wol10}]\label{paulapath}
For every graph $G$, every abelian group~$\Gamma$, every undirected $\Gamma$-labelling~$\gamma$,
every set $A\subseteq V(G)$ and every integer $k$, the graph $G$ contains $k$ disjoint
non-zero $A$-paths with respect to $\gamma$
or a set $X\subseteq V(G)$ of size $O(k^4)$ such that 
$G-X$ does not contain any non-zero $A$-path with respect to~$\gamma$.
\end{theorem}

\begin{lemma}\label{bipcase}
Let $G$ be a graph, and let $A\subseteq V(G)$ be a set such that $G-A$ is bipartite.
Then, for every integer $k$ the graph $G$ either contains 
$k$ disjoint zero $A$-paths 
or there is a set $X\subseteq V(G)$ of size $O(k^4)$ such that $G-X$ does not 
contain any zero $A$-path.
\end{lemma}
\begin{proof}
Let $V_1,V_2$ be a bipartition of $G-A$. 
Starting with $G$ we define a graph $G'$ by 
replacing each vertex $a\in A$ by two vertices 
$a_1,a_2$, where we make $a_i$ adjacent to $N(a)\cap V_{3-i}$ for $i=1,2$. 
Set $A_i=\{a_i:a\in A\}$, and observe that the  graph $G'$ is bipartite. 

Clearly, if  for one $i\in\{1,2\}$ we find
 $k$ disjoint zero $A_i$-paths in $G'-A_{3-i}$ then there are 
also $k$ disjoint zero $A$-paths in $G$. On the other hand, if for $i=1,2$ 
we find vertex sets $X_i$ such that every zero $A_i$-path
in $G'-A_{3-i}$ meets $X_i$ then {
\[
 \{a\in A:a_1\in X_1\text{ or }a_2\in X_2\} \cup (X_1\cup X_2)\sm (A_1\cup A_2)
\]}%
meets every zero $A$-path of $G$. 
(Here, we use that zero $A$-paths have even length.) 

Thus, when proving the statement of the lemma, 
we may assume that $G$ is bipartite and that 
$A$ is completey contained in one of the bipartition
classes of $G$. That means that every $A$-path has even length.

We define an undirected $\mathbb Z_4$-labelling on $G$ by setting $\gamma(e)=0$ if 
$e$ is an edge that is incident with a vertex in $A$, and by setting $\gamma(e)=1$
otherwise. Let $P$ be an $A$-path. Then 
\[
\gamma(P)= 0\cdot 2+1\cdot(|E(P)|-2) = |E(P)|-2,
\]
where we calculate in $\mathbb Z_4$.
Thus $\gamma(P)\in \{0,2\}$ as every $A$-path has even length. Moreover, if $\gamma(P)=2$ then 
$|E(P)|\equiv 0\text{ mod }4$, and if $\gamma(P)=0$ then $|E(P)|\equiv 2\text{ mod }4$. 
Thus the $A$-paths of length congruent to~$0$ modulo~$4$ are precisely the non-zero $A$-paths
with respect to the labelling~$\gamma$. An application of Theorem~\ref{paulapath} finishes the proof.
\end{proof}

Let $B$ be a block of a graph $G$ and let $b_1,..,b_l$ be the cutvertices of $G$ that are contained in $B$. 
A component of $G-(V(B)\sm\{b_1,.,,b_l\})-E(B)$ is a  \emph{$B$-bridge}.\footnote{Normally, the definition 
of a bridge is {a bit different;} see for instance Bondy and Murty~\cite[Chapter~10.4]{BM08}.}

\begin{lemma}\label{bipblockcase}
There is a function $\pfun$ such that:
If $A$ is a vertex set in a graph $G$, if 
 $B$ is a bipartite block of $G-A$,  if $B_1,\ldots, B_\ell$ are the $B$-bridges
in $G-A$, and 
if there does not exist any zero $A$-path whose interior is contained in one of the $B$-bridges $B_1,\ldots, B_\ell$
then, for  every integer $k$, the graph $G$ either contains $k$ disjoint zero $A$-paths
or a set $X\subseteq V(G)$ of size $|X|\leq\pfun(k)$ such that $G-X$ does not 
contain any zero $A$-path.
\end{lemma}

\begin{proof}
{
Let $\pfun$ be the size of the hitting set $X$ in Lemma~\ref{bipcase}.}
Clearly, we may assume $G-A$ to be connected. 
For every~$i$, let $b_i$ be the common vertex of $B$ and $B_i$. 

Starting from $G[A\cup B]$ we form a graph $G^*$ as follows. 
For every $B_i$ take a disjoint copy $T_i$ of the tree in Figure~\ref{replacetree}
on the left, and identify its root $r$ with $b_i$. 
Now, for every $a\in A$ make $a$ adjacent to 
the copy of $t_p$, $p\in\{0,1,2,3\}$, in $T_i$  if 
there is an $a$--$b_i$ path $Q_{a,i}^p$ contained in $G[B_i\cup\{a\}]$ 
that has length congruent to~$p$ modulo~$4$.
Observe that $G^*-A$ is bipartite.

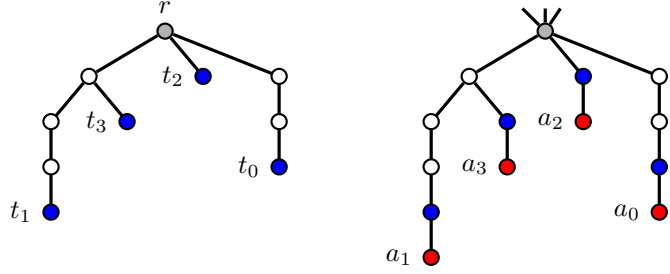
\begin{figure}[ht]
\centering
\begin{tikzpicture}[yscale=0.6]

\tikzstyle{avx}=[hvertex, fill=red]
\tikzstyle{tvx}=[hvertex, fill=blue]

\node[hvertex,fill=hellgrau,label=above:$r$] (r) at (0,0){};
\node[hvertex] (z) at (-1,-1){};
\draw[hedge] (r) -- (z);


\draw[hedge] (z) -- ++(-0.5,-1) node[hvertex] {} -- ++(0,-1) node[hvertex] {} -- ++(0,-1)  node[tvx,label=left:$t_1$]{}; 

\draw[hedge] (z) -- ++(0.5,-1)  node[tvx,label=left:$t_3$] {};

\draw[hedge] (r) -- ++(0.5,-1)  node[tvx,label=left:$t_2$] {};

\draw[hedge] (r) -- ++(1.5,-1) node[hvertex] {} -- ++(0,-1) node[hvertex] {} -- ++(0,-1) node[tvx,label=left:$t_0$]{};

\begin{scope}[shift={(5,0)}]
\node[hvertex,fill=hellgrau] (r) at (0,0){};
\node[hvertex] (z) at (-1,-1){};
\draw[hedge] (r) -- (z);

\draw[hedge] (r) edge ++(-0.3,0.5) edge ++(0,0.5) edge ++(0.2,0.5);

\draw[hedge] (z) -- ++(-0.5,-1) node[hvertex] {} -- ++(0,-1) node[hvertex] {} -- ++(0,-1) node[tvx]{} -- ++(0,-1) node[avx,label=left:$a_1$]{}; 

\draw[hedge] (z) -- ++(0.5,-1) node[tvx] {} -- ++(0,-1) node[avx,label=left:$a_3$] {};

\draw[hedge] (r) -- ++(0.5,-1) node[tvx] {} -- ++(0,-1) node[avx,label=left:$a_2$] {};

\draw[hedge] (r) -- ++(1.5,-1) node[hvertex] {} -- ++(0,-1) node[hvertex] {} -- ++(0,-1) node[tvx]{} -- ++(0,-1) node[avx,label=left:$a_0$]{}; 
\end{scope}

\end{tikzpicture}
\caption{Left: the tree we use to replace the $B_i$. Right: the tree with some vertices from $A$ 
attached to its leaves. Note that no zero $A$-path lies in the tree (plus $A$).}\label{replacetree}
\end{figure}

We first observe that
\begin{equation}\label{PP*}
\begin{minipage}[c]{0.8\textwidth}\em
for every zero $A$-path $P$ in $G$ there is a zero $A$-path $P^*$ in $G^*$, 
and vice versa, such that $P\cap (B\cup A)=P^*\cap (B\cup A)$.  
\end{minipage}\ignorespacesafterend 
\end{equation} 
Indeed, let  $P$ be a zero $A$-path with endvertices $a,a'$. 
Assume that a neighbour $b$ of $a$ in $P$ 
lies in some $B_i$. Then, as $P$ meets $B$ by assumption, the path $P$ passes through $b_i$.
Let the subpath $aPb_i$ have length $p$ modulo~$4$. Then we replace in $P$ 
the subpath $bPb_i$ by the path in {the copy} $T_i$ between the copy of $t_p$ and $b_i$.
We do this at both ends, if necessary, of $P$ and obtain $P^*$ in this way. 
That $P^*$ is indeed a path is due to the fact that if $P$ meets $b_i$ then it 
leaves $B_i$ there, as the interior of $P$ cannot be contained in $B_i$ by assumption.
Observe that the length modulo~$4$ did not change. The other direction is 
similar but uses the paths $Q^p_{a,i}$ and the observation that 
also in $G^*$ the interior of a zero $A$-path cannot be contained in any $T_i$.
(See Figure~\ref{replacetree} on 
the right.) 

Next we claim that
\begin{equation}\label{disjointP}
\emtext{
if $P_1^*,P_2^*$ are disjoint zero $A$-paths in $G^*$ then $P_1,P_2$ are also disjoint.
}
\end{equation}
By~\eqref{PP*}, $P_1$ and $P_2$ do not intersect in $B$. Moreover, as both need to meet $B$
by assumption they cannot meet in any $B_i$ either as then they would both need to contain $b_i\in V(B)$.
This proves~\eqref{disjointP}.

As a consequence of~\eqref{PP*} and~\eqref{disjointP}, 
we see that if $G^*$ has $k$ disjoint zero $A$-paths then so does $G$. 
By Lemma~\ref{bipcase}, we may therefore assume that $G^*$ admits a hitting set $X^*$
of size $|X^*|=\pfun(k)$. We define
\[
X=(X^*\cap V(G))\cup\{b_i:X^*\cap V(T_i)\neq\emptyset,\, i=1,\ldots,\ell\}
\]
and observe that $|X|\leq |X^*|$. 

Consider a zero $A$-path $P$ of $G$ that does not meet $X\cap B$. As its corresponding path $P^*$
is a zero $A$-path in $G^*$, by~\eqref{PP*}, it is  met by $X^*$. Thus $X^*$ contains some vertex 
from $T_i\cap P^*$ for some $i$. Then $b_i\in X$.
Moreover, as $P$ must meet $B$, it follows from $P\cap B=P^*\cap B$ that $P^*$, and thus $P$ as well,
passes through $b_i$. This implies that $P$ meets $X$. Consequently, $X$ is a hitting set in $G$.
\end{proof}

\section{Proof of Theorem~\ref{zeroepp}}

We prove Theorem~\ref{zeroepp} over the course of this section.
Before we start, we specify {the function that bounds the size of} the hitting set. For this, we 
will use a number of functions, such as $\rsfun(t)$, where we remind the reader 
that their index denotes the theorem or lemma in which the function is defined. 
One of these functions will only be defined later. We will make sure that it only depends on~$k$.

First, for every positive integer $k$, define $t^*(k)$ such that
\begin{equation}\label{choicet}
 t^*(k)\geq 2600k^3\emtext{ and }t^*(k)\geq \bfun(k)
\end{equation}

Next, define $g(k)$ such that
\begin{equation}\label{choiceg}
  g(k)\geq \fptfun(t^*(k))+2\emtext{ and } g(k)\geq \rsfun(\fptfun(t^*(k)))
\end{equation}

Third, we define $f(k)$, the size of the hitting set, such that
\begin{equation}\label{choicef}
 f(k)\geq\pfun(k)\emtext{ and }f(k)\geq 3g(k)+2f(k-1)+10
\end{equation}

Suppose that for some $k$, there is a graph that does not contain $k$ disjoint zero $A$-paths
and that does not admit a hitting set of size at most $f(k)$. 
We fix for the rest of this section a smallest such $k$ and  a graph $G$ such that
\begin{equation}\label{counterG}
\begin{minipage}[c]{0.8\textwidth}\em
every graph $H$ contains
either $k-1$ zero $A$-paths or a hitting set of size at most $f(k-1)$,
but $G$ does not contain $k$ zero $A$-paths and does not contain a hitting
set of size at most $f(k)$ either. 
\end{minipage}\ignorespacesafterend 
\end{equation} 

We start the proof of Theorem~\ref{zeroepp} by showing
 that $G-A$ admits a large tangle, and then a large wall, and therefore
satisfies the conditions of Theorem~\ref{oddflatwallthm}.
We remark that this approach is a fairly useful and also 
common approach for Erd\H os-P\'osa type questions. 
Similar arguments have been made, for instance, by 
Wollan~\cite{Wol11} and Liu~\cite{Liu17}. If $G$ and $H$ are two graphs,
we write $G-H$ for the induced subgraph $G-V(H)$ of $G$.

\newcommand{\eptan}{\ensuremath{\mathcal T_{\text{\rm EP}}}}

\begin{lemma} \label{largetanglem}
The graph $G-A$ admits a tangle \eptan\ of order $g(k)$ such that for each 
 separation $(C,D) \in \eptan$ every zero $A$-path has to intersect~$D-C$. 
Additionally, the graph $G[A \cup (D -C)]$ contains a zero $A$-path.
\end{lemma}

\begin{proof}
Consider a separation $(C,D)$ in $G - A$ of order at most $g(k)-1$. 
We first prove:
\begin{equation}\label{exclm}
\begin{minipage}[c]{0.8\textwidth}\em
Exactly one of the graphs $G[A \cup C]$ and $G[A \cup D]$
contains a zero $A$-path.
\end{minipage}\ignorespacesafterend 
\end{equation} 

We define an auxiliary graph $H$ on the vertex set $A$, 
where $a_1,a_2 \in A$ are adjacent if there is a zero $A$-path in $G-(C\cap D)$ 
with endvertices $a_1$ and $a_2$. If $H$ has no matching of size five then 
there exists a vertex cover $X$ of at most ten vertices. In that case, $X\cup (C\cap D)$ would
combine to a hitting set for $G$, which is impossible as $10+|C\cap D|\leq g(k)+10\leq f(k)$
by~\eqref{choicef}.

Thus, $H$ has a matching of size five, which implies that at least one of 
$G[A \cup (C - D)]$ and $G[A \cup (D - C)]$, say the latter, contains three zero $A$-paths 
such that their endvertices, $\{a_1,a_2\}$, $\{a_3,a_4\}$ and $\{a_5,a_6\}$ are all distinct.

Now, contrary to~\eqref{exclm}, suppose that also  $G[A \cup C]$ contains a zero $A$-path~$P$.
Then, the endvertices $a,a'$ of $P$ are disjoint from one of the 
pairs $\{a_1,a_2\}$, $\{a_3,a_4\}$ and $\{a_5,a_6\}$, let us assume from $\{a_1,a_2\}$. 
We form two  subgraphs $G_1=G[(A-\{a_1,a_2\})\cup C]$ and $G_2=G[(A-\{a,a'\})\cup (D-C)]$ of $G$
that are disjoint outside $A$, 
and use~\eqref{counterG} for each of them. If $G_1$ contains $k-1$ 
disjoint zero $A$-paths then we find, together with the zero 
$A$-path contained in $G[\{a_1,a_2\} \cup (D-C)]$,
$k$ disjoint zero $A$-paths, which we had excluded in~\eqref{counterG}. 
Similarly, $G_2$ cannot contain $k-1$ disjoint zero $A$-paths.
Thus, for $i=1,2$ there must be a hitting set $X_i$ of size at most $f(k-1)$ in $G_i$.
But then, the set $X_1\cup X_2\cup (C\cap D)\cup\{a_1,a_2,a,a'\}$ 
is a hitting set for $G$ of size at most
\[
2f(k-1)+g(k)+4\leq f(k),
\]
by~\eqref{choicef},
which means that we are done. Therefore,~\eqref{exclm} is proved. 

We use~\eqref{exclm} to define a tangle~\eptan\ of order $g(k)$: 
if $(C,D)$ is a separation of $G-A$ of order less than $g(k)$ then include $(C,D)$ in~\eptan\
if $G[A \cup D]$ contains a zero $A$-path. If \eptan\ defines a tangle, we are 
done: indeed, note  that for $(C,D)\in\eptan$ 
 any zero $A$-path needs to meet $D-C$ as otherwise there would be a
zero $A$-path contained in $G[A\cup (C\cap D)]$, in contradiction to~\eqref{exclm}. 
Moreover, $C\cap D$ cannot be a hitting set as it is too small, which means there must 
be a zero $A$-path that avoids $C\cap D$ and thus is contained in $G[A\cup (D-C)]$.

To see that $\eptan$ is a tangle, we still have to verify that conditions~(T\ref{tangle2})
and~(T\ref{tangle3}) of the tangle definition are satisfied. 
For~(T\ref{tangle2}) suppose that  $V(C) = V(G - A)$ for some $(C,D)\in\eptan$. 
By~\eqref{exclm}, there is a zero $A$-path in $G[A \cup (D - C)]$.
As, on the other hand, $D - C = \emptyset$,
it follows that this  path has to be completely contained in $G[A]$, which is impossible
as the only  $A$-paths in $G[A]$ have length~$1$. 

Suppose that~(T\ref{tangle3}) is false, i.e.\ suppose that there are $(C_1,D_1)$, $(C_2,D_2)$, $(C_3,D_3)\in\eptan$
with  $C_1 \cup C_2 \cup C_3 = G - A$.
Consequently, every  zero $A$-path must meet $C_i$ for some $i$, which implies that 
the set $\bigcup_{i=1}^3 (C_i \cap D_i)$ is a hitting set of size at most $3g(k)\leq f(k)$,
by~\eqref{choicef},
which again means that we are done. 
\end{proof}

A \emph{windmill} is a graph consisting of the union of three paths $P_1,P_2,P_3$ and of three cycles $C_1,C_2,C_3$
such that
\begin{itemize}
\item the paths $P_1,P_2,P_3$ share an endvertex $x$, but are otherwise disjoint;
\item the cycles $C_1,C_2,C_3$ have odd lengths;
\item $C_i\cap P_i$ is a path of length at least~$1$ for $i=1,2,3$; and 
\item $C_i$ is disjoint from $\bigcup_{j\neq i}P_j\cup C_j$.
\end{itemize}
If $a_i$ is the endvertex of $P_i$ that is not equal to $x$, then $a_1,a_2,a_3$ are the \emph{tips}
of the windmill.

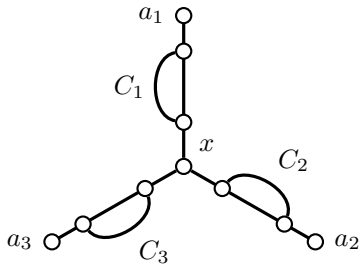
\begin{figure}[ht]
\centering
\begin{tikzpicture}
\def\step{2}
\node[hvertex,label=above right:$x$] (x) at (0,0){};
\draw[hedge] (x) to node[hvertex,near start] (a1){} node[hvertex,near end] (a2){} (0,\step) node[hvertex,label=left:$a_1$](a){};
\draw[hedge] (x) to node[hvertex,near start] (b1){} node[hvertex,near end] (b2){} (-30:\step) node[hvertex,label=right:$a_2$](b){};
\draw[hedge] (x) to node[hvertex,near start] (c1){} node[hvertex,near end] (c2){} (210:\step) node[hvertex,label=left:$a_3$](c){};
\draw[hedge,bend left=80] (a1) to node[midway,auto]{$C_1$} (a2);
\draw[hedge,bend left=80] (b1) to node[midway,auto]{$C_2$} (b2);
\draw[hedge,bend left=80] (c1) to node[midway,auto]{$C_3$} (c2);
\end{tikzpicture}
\caption{A windmill}\label{windfig}
\end{figure}

\begin{lemma}\label{clawlem}
Let $W$ be a windmill with tips $a_1,a_2,a_3$. Then $W$ contains a zero $\{a_1,a_2,a_3\}$-path.
\end{lemma}
\begin{proof}
For $i=1,2,3$ the windmill contains two distinct $a_i$--$x$~paths $P_{i1},P_{i2}$, 
each using a different path along
the cycle $C_i$. Since the length of $C_i$ is odd it follows that 
one of $P_{i1},P_{i2}$ has odd length and the other even length. 
In particular, the lengths are different  modulo~$4$. 

Let $\alpha,\beta\in\mathbb Z_4$ 
be the lengths of $P_{11},P_{12}$ modulo~$4$. Then $\alpha\neq\beta$. 
If  $W$ {contains} a zero $a_1$--$\{a_2,a_3\}$~path then we are done. So assume that is not the case. 
Thus none of the paths $P_{21},P_{22},P_{31},P_{32}$ has length congruent to $-\alpha$ or to $-\beta$. 
{If} $\{\gamma,\delta\}=\mathbb Z_4\setminus\{-\alpha,-\beta\}$ then the paths $P_{21},P_{22}$ have 
lengths congruent to $\gamma,\delta$ modulo~$4$,
and this is also the case for $P_{31},P_{32}$. 
Since the paths $P_{i1},P_{i2}$ have different parity, 
either $\gamma$ or $\delta$, say $\gamma$, is one of $0,2$. 
By combining the two paths of length congruent to $\gamma$ we obtain a zero $a_2$--$a_3$~path. 
\end{proof}

By Lemma~\ref{largetanglem} and Theorem~\ref{robseywalltangle}, the graph $G-A$ contains a 
large wall such that its induced tangle is a truncation of \eptan. 
Thus, one of the four cases of Theorem~\ref{oddflatwallthm} will apply to $G-A$. 

We first consider the case where there is a large odd $K_t$-model in $G-A$. 
Equivalently, this means that $G$ contains an odd $K_t$-model that is disjoint from $A$. 
We say a collection of disjoint paths \emph{nicely link} to a $K_t$-model $\pi$ 
if each path intersects exactly one branch set of $\pi$ and different paths intersect different branch sets. Moreover, a path from $A$ to a branch set of $\pi$ is called an $A$-$\pi$-path.

We also need the notion of submodels of models. If $H'$ is a subgraph of $H$ and there is an $H$-model $\pi$ in $G$, then the $H'$-submodel of $\pi$ is the restriction of $\pi$ to $H'$.

\begin{lemma} \label{pathstoclique}
Let $\ell,t \in \N$ and $t \geq 3\ell$ and let $\pi$ be a $K_t$-model in a graph {$H$}. For a set of vertices $A \subseteq {V(H)}$ there is a $K_{t-2\ell}$-submodel $\eta$ such that there are either $\ell$ disjoint 
$A$--$\eta$~paths that nicely link to $\eta$ or a set $X$ of at most $2\ell-1$ vertices that separates 
$A$ from $\eta$.
\end{lemma}

\begin{proof}
Let the vertex set of $\pi$ be $\{v_1,\ldots,v_T\}$.
For every vertex $v_i$, 
we select a vertex $u_i\in V(\pi(v_i))$.
Let $U=\{u_1,\ldots,u_T\}$.
By Menger's theorem,
there are either $2t$ disjoint $A$--$U$-paths or a set $X$ of size less than $2t$ that separates $A$ from $U$.

Assume first that there is a set $X$ of size $|X|<2t$ that separates $A$ from $U$. 
Then,
every branch set $\pi(v_i)$ of $\pi$ that is not separated from $A$ by $X$
has to meet $X$ --- this is because branch sets are connected. Thus, at least $T-2t$
of the 
branch sets of $\pi$ are separated from $A$ by $X$; taking these yields the desired $K_{T-2t}$-subexpansion~$\eta$.

Second, we treat the case when there are $2t$ disjoint $A$--$U$-paths $P_1,\ldots,P_{2t}$.
Denote their union by $\cP$.
Clearly, we can assume that if $u_i$ is not contained in any path in $\cP$,
then $\pi(v_i)$ is disjoint from any path in $\cP$.

We say that a path $P\in\cP$ with endvertex $u_i\in U$ 
has a \emph{private end} $v_i$ if $\pi(v_i)$ meets no path in $\cP-P$.
Let $\text{PE}(\cP)$ be the set of private ends of the paths in $\cP$.
We can partition $P$ into disjoint subpaths $Q_1,\ldots, Q_k$ such that $Q_j$ only intersects exactly one vertex of $\pi$.
If $k$ is minimal with respect to this property, we say $P$ visits $\pi$ (exactly) $k$ times and write $\text{vis}(P)=k$.
Let $\text{vis}(\cP)=\sum_{P\in \cP}\text{vis}(P)$.

We next choose $\cP$ such that 
\begin{enumerate}[\rm (i)]
	\item the paths in $\cP$ intersect at most $2t+|\text{PE}(\cP)|$ branch sets of $\pi$
	\item subject to (i), $\text{pe}(\cP)$ is maximal, and
	\item subject to (ii), $\text{vis}(\cP)$ is minimal.
\end{enumerate}
Clearly, this is a valid choice. 
Suppose that  $|\text{PE}(\cP)|< t$, i.e.\ that $\cP$ has fewer than $t$ private ends.
As $T\geq 3t$, 
we may assume that no path in $\cP$ intersects $\pi(v_1)$.

Denote by $\text{PE}^-(\cP)$ the set of vertices $v_j$ of $\pi$ such that there is a path $P\in\cP$
with a private end $v_i\neq v_j$ such that $P$ meets $\pi(v_j)$,  in a vertex $u'_j$ say, 
and such that the end segment $u'_jPu_i$ of $P$ is disjoint from any branch set of $\pi$, 
except for $\pi(v_i)$ and $\pi(v_j)$. That is, $P$ passes through $\pi(v_j)$ and then ends
in $\pi(v_i)$, its private end, without meeting other branch sets in between. 

Obviously $|\text{PE}^-(\cP)|\leq |\text{PE}(\cP)|$. Thus, $|\text{PE}^-(\cP)|+ |\text{PE}(\cP)|<2t$.
As, in addition, $\cP$ meets at least $2t$ branch sets, it follows that there is a 
vertex in $\{v_1,\ldots, v_T\}\sm (\text{PE}^-(\cP)\cup \text{PE}(\cP))$ such that its branch set
meets a path in $\cP$. We may assume that $v_2$ is like that. 

Let $Q$ be the unique $u_1$--$u_2$-path in $\pi(v_1)\cup \pi(v_2)\cup \pi(v_1v_2)$.
Let $P\in \cP$ be the first path that intersects $Q$, say in $q$, seen from $u_1$.
Let $a$ be the  endvertex of $P$ in $A$.
Observe that replacing $P$ by $aPqQu_1$ contradicts the choice of $\cP$,
as either $|\text{PE}(\cP)|$ increases, 
or $\text{vis}(\cP)$ decreases while $\text{pe}(\cP)$ does not change.

We obtain the desired the set of paths by selecting $t$ paths from $\cP$ 
that each have a private end. Indeed, by~(i), the $t$ paths intersect at most $2t$ branch sets
besides the ones of their private ends. Omitting these from $\pi$ leads to a
$K_{T-2t}$-subexpansion $\eta$ such that the paths nicely link to it.
\end{proof}

\begin{lemma}\label{cliqueminor}
{
Let $t\geq 16k$
If $G$ contains 
an  odd $K_t$-model $\pi$ that is disjoint from $A$ and 
whose induced tangle is a truncation of \eptan\ 
then $G$ contains $k$ disjoint zero $A$-paths.}
\end{lemma}

\begin{proof}
We apply Lemma~\ref{pathstoclique} with $\ell = 3k$ to $\pi$ and $A$ in $G$. 
As $t \geq 3 \cdot 3k$, we find 
a $K_{t-6k}$-submodel $\eta$ of $\pi$ such that there are either $3k$ disjoint 
$A$--$\eta$~paths that nicely link to $\eta$, or a set $X$ of $6k-1$ vertices that separates $A$ from $\eta$. Since $t -6k \geq 10k$, we {deduce} that $\eta$ has at least $10k$ branch sets.  

Suppose first that Lemma~\ref{pathstoclique} yields such a set $X$. 
Denote by $C'$ the set of all vertices that can be reached from $A$ in $G-X$, and let $D'$ be the 
set of those vertices that cannot be reached. Put $C=G[C'\cup X]$ and $D=G[D'\cup X]$ and observe
that $(C,D)$ is a separation of order~$|X|\leq 6k-1$. In particular, either $(C,D)\in\mathcal T_\pi$
or $(D,C)\in\mathcal T_\pi$ as $ 6k-1 < 6k = \lceil \frac{3\cdot 6k}{3} \rceil \leq \lceil \frac{2\cdot t}{3} \rceil$. Pick one of the at least $10k$ branch sets $Y$ of $\eta$ that is disjoint from $X$ (which has 
size less than $6k$), and observe that $Y\subseteq D'=D-C$ by definition of $X$. Thus $(C,D)\in\mathcal T_\pi$.
On the other hand, $\mathcal T_\pi$ is a truncation of \eptan, which means, by Lemma~\ref{largetanglem},
that $G[A\cup (D-C)]$ contains a zero $A$-path. This is impossible, however, as $X$ separates $D-C$ from $A$.
 
Therefore, Lemma~\ref{pathstoclique} yields a set $\mathcal P$ of $3k$ disjoint $A$--$\eta$~paths
that nicely link to $\eta$. 
Choose $10k$ of the branch sets of $\eta$ so that they include the ones in which the paths end, 
and partition them into groups of ten branch sets each, such that always exactly three of them 
{in each group}
contain an endvertex of one of the paths. Consider such a group $Y_1,\ldots, Y_{10}$, where 
we assume that for $i=1,2,3$ the path $P_i\in\mathcal P$ ends in $Y_i$. 
We will construct a windmill with its tips in $A$
in the (induced graph on the) union of $Y_1,\ldots,Y_{10}$ together with the paths $P_1,P_2,P_3$. Note that 
such a windmill meets $A$ only in its tips as $\pi$ and thus $\eta$ is disjoint from $A$. 
 In total, we will thus find $k$ disjoint windmills that meet $A$ precisely in their tips. 
Lemma~\ref{clawlem} then yields $k$ disjoint zero $A$-paths. 

To construct such a windmill, observe that  $\eta$ is (or rather contains) an odd $K_{10k}$-model.
Thus, there is an odd cycle $C_1$ contained in the induced graph on $Y_1\cup Y_4\cup Y_5$.
There is, moreover, a $P_1$--$C_1$~path $Q_1$ starting in the endvertex of $P_1$ in $C_1$ contained in $Y_1$,
and there is a  $C_1$--$Y_{10}$~path $R_1$ that is contained
in $Y_4\cup Y_{10}$, and that then meets $Y_{10}$ only in its final vertex, which we denote by $r_1$.
Note that the endvertices of both these paths on $C_1$ are distinct. 
Set $T_1=P_1\cup Q_1\cup C_1\cup R_1$ and observe that $T_1-r_1$ is contained in the induced graph on $Y_1\cup Y_4\cup Y_5\cup P_1$.
Using $Y_2,Y_6,Y_7,P_2$ and $Y_3,Y_8,Y_9,P_3$, we find  similar graphs $T_2,T_3$ such that $T_1,T_2,T_3$
meet at most in their final vertices in $Y_{10}$. Pick a subtree $T$ of $Y_{10}$ with leaves $r_1,r_2,r_3$
and observe that $T_1\cup T_2\cup T_3\cup T$ is a windmill with tips in $A$.
\end{proof}

Next, we treat outcome~(b) of Theorem~\ref{zeroepp}. 
This is split into two subcases: first we deal with the case when there is a large wall whose 
bricks are all odd, and  then with the case when there is a bipartite wall with an odd linkage.

Let $W$ be a wall with nails $N$. 
A set $\mathcal P$ of {disjoint} paths \emph{nicely links to $W$} if each path $P\in\mathcal P$
is contained in $G-(W-N)$ and ends in a (distinct) nail of~$W$.

\begin{lemma}\label{linklem}
Let $r,t$ be positive integers with $r\geq t$.
Let {$H$} be a graph containing a wall $W$ of size at least $4tr$, and let $A\subseteq {V(H)}$ 
be disjoint from $W$. 
Then $W$ has a subwall $W_1$ of size at least $r$ such that there is either a 
vertex set $X$ of size $|X|<3t^2$ that separates $A$ from $W_1$ or there is a 
set $\mathcal P$ of $t$ disjoint paths that nicely links $A$ to $W_1$. 
\end{lemma}

\begin{proof}
Let $W_0$ be a $t$-contained subwall of size at least $2tr$ of $W$, and let $B$ be the 
set of branch vertices of $W_0$. Assume first that there is a vertex set $X$
of size $|X|\leq 3t^2-1$ that separates $A$ from $B$. As $W_0$ has size at least $ 2t r$,
it follows that $W_0$ has $4t^2$ disjoint subwalls of size $r$. One of these, $W_1$ say, is disjoint from $X$.
Then $W_1$ is separated from $A$ by $X$.

Thus, assume that there is no such set $X$. By Menger's theorem, we find a set $\mathcal Q_1$
of $3t^2$ $A$--$B$~paths. 
Choose $\mathcal Q_1$ such that $\sum_{Q\in\mathcal Q_1}|E(Q)\sm E(W_0)|$ is minimal.
Denote by $B_1$ the set of endvertices in $W_0$
of the paths in $\mathcal Q_1$.

By choice of $\mathcal Q_1$, whenever a path $Q\in\mathcal Q_1$ meets
a branch vertex of $W_0$ it ends there; and if a path $Q\in\mathcal Q_1$ meets a subdivided edge 
in its interior then at least one of its endvertices, which are branch vertices, must be in $B_1$. 
(Otherwise, we can reroute.)
Thus, if a subwall of $W_0$ is disjoint from $B_1$ then it is disjoint from any path in $\mathcal Q_1$. 

Since $W_0$ has size at least $2tr$ there is a subwall $W_1$ 
of size $r$ that is disjoint from $B_1$, and then from any path in $\mathcal Q_1$. Let $N_1$ be the set of nails
of $W_1$ (with respect to $W$). 
We claim that $A$ cannot be separated from $N_1$ in $G-G-(W_1-N_1)$ by fewer than $t$ vertices. 

For this, 
let $Y$ be a vertex set in $H-(W_1-N_1)$ of size at most $t-1$. 
As $W_0$ is $t$-contained in $W$
there is a cycle $C$ in $W-W_0$ that is disjoint from $Y$.
Moreover, as $|\mathcal Q_1|= 3t^2$, 
there is a subset $\mathcal Q_2$ of $\mathcal Q_1$ of size at least $2t^2$ that is disjoint from $Y$. 
Let $B_2$ be the set of endvertices in $W_0$ of the paths in $\mathcal Q_2$. 

Each vertex in $B_2$ lies in a horizontal and 
in a vertical path of $W_0$, and at most two lie in the same pair of vertical and horizontal path.
Since, moreover,  $|B_2|\geq 2t^2$ there is a horizontal or vertical path $P$ of $W$
that meets some vertex from $B_2$ and is disjoint from $Y$. Note that $P$ has a $B_2$--$C$~subpath $P'$
that is disjoint from $W_1$. Finally, there are at least $r\geq t$ disjoint subpaths of vertical paths of $W$
that link $C$ to $N_1$. One of these, $R$ say, is disjoint from $Y$. Then $\mathcal Q_2$, $P'$ and $R$
show that $N_1$ cannot be separated from $A$ in $H-(W_1-N_1)$ by fewer than $t$ vertices. 
Thus, by Menger's theorem, there is 
a set $\mathcal P$ of $t$ disjoint paths that nicely links $A$ to $W_1$. 
\end{proof}

\begin{lemma}\label{linklem2}
Let $r,t$ be positive integers such that 
$r\geq 3t^2$, and 
let $W$ be a wall of size at least $4tr$ in $G-A$ 
such that its induced tangle $\mathcal T_W$ in $G-A$ is a truncation of $\eptan$.
Then, there is an $r$-subwall $W_1$ of $W$ and a set of $t$ disjoint $A$--$W_1$~paths 
that nicely link to $W_1$.
\end{lemma}
\begin{proof}
Applying Lemma~\ref{linklem} yields a subwall $W_1$ of $W$ of size~$r$.
{Suppose}
there is a set $X$ of fewer than $3t^2$ vertices that separates $W_1$ from $A$ in $G$. 
We construct a separation of $G-A$ by putting all vertices that
 are reachable from $A$ in $G-X$ into $C'$ and those that are not into $D'$. 
Set $D=G[D' \cup X]-A$ and $C=G[C' \cup X]-A$. 
Since the size of $W_1$ is $r \geq 3t^2 > |X|$, it follows that either $(C,D)$ or $(D,C)$ 
lies in the tangle $\mathcal T_{W_1}$ that $W_1$ induces in $G-A$.
As $r\geq 3t^2$, some horizontal path of $W_1$ is disjoint from $X$ and thus lies in $D'$, 
as $X$ separates $A$ from $W_1$. Thus $(C,D)\in\mathcal T_{W_1}$, which implies $(C,D)\in\mathcal T_W$
and then $(C,D)\in\eptan$.
Lemma~\ref{largetanglem}, however, states that some zero $A$-path is contained in 
 $G[A \cup (D-C)]$, which is impossible as $X$ separates $A$ from $D-C$. 

This contradiction shows that the subwall $W_1$ given by Lemma~\ref{linklem} comes 
with a set of $t$ disjoint $A$--$W_1$~paths that nicely link to $W_1$.
\end{proof}

\begin{lemma}\label{nonzerowallminor}
Let $W$ be a wall of size $2600k^3$ in $G-A$
such that its bricks are odd cycles and such that $\mathcal T_W$ is a 
 truncation of $\eptan$. 
Then there are $k$ disjoint zero $A$-paths.
\end{lemma}
\begin{proof}
We start by using Lemma~\ref{linklem2} on $W$ with $t =6k$ and $r=3t^2$. 
As a result, we obtain a  subwall $W_1$ of $W$ of size {at least} $r>6k$
 and a set of $6k$ disjoint $A$--$W_1$~paths
that nicely link to $W_1$. By choosing a subset $\mathcal P$ of $3k$ of these paths we can ensure that 
the bricks $B_P$, for $P\in\mathcal P$, in which they end are pairwise disjoint. 
Moreover, as $W_1$ is a  subwall of $W$, its bricks are odd cycles as well. 

Let $Q$ be the third horizontal path in $W_1$ from the top. In particular, $Q$ is disjoint from each
of the bricks $B_P$, $P\in\mathcal P$. For each $P\in\mathcal P$ connect $B_P$ to $Q$ via a
subpath $Q_P$ of a vertical path, such that all $Q_P$ are pairwise disjoint, and also disjoint from all 
$B_{P'}$ for $P'\in\mathcal P$ with $P'\neq P$. Then, if we order the paths in $\mathcal P$
with respect to their endvertex on the top vertical path of $W_1$, 
for each three consecutive paths in $\mathcal P$, the union of the
sets $P\cup B_P\cup Q_P$ together with 
a subpath of $Q$ forms a windmill with tips in $A$
(note that $B_P$ is an odd cycle). That windmill, moreover, does not meet $A$
outside its tips as $W_1$ is disjoint from $A$. In this way, we obtain $k$ disjoint such windmills,
and then, by Lemma~\ref{clawlem}, $k$ disjoint zero $A$-paths.
\end{proof}

Now we deal with outcome~(b.ii) of Theorem~\ref{oddflatwallthm}. For this we first see that 
we can get rid of any interference between paths linking $A$ to a wall and a linkage of the wall. 

\begin{lemma}\label{twotypespaths}
Let $t$ be a positive integer, and let $H$ be a graph containing 
three vertex sets $A,B,X$.
If $H$ contains $2t$ disjoint $A$--$X$~paths $Q_1,\ldots, Q_{2t}$ and $t$ disjoint $B$--$X$~paths $R_1,\ldots,R_t$
then $H$ contains $2t$ disjoint paths $P_1,\ldots,P_{2t}$ such that $P_i$ is a $A$--$X$~path for $i\in [t]$
and a $B$--$X$~path for $i\in [2t]\sm [t]$.
Moreover, $\{P_1,\ldots,P_{t}\}\subseteq \{Q_1,\ldots,Q_{2t}\}$
and $P_i\subseteq \bigcup_{j\in[2t]}Q_j\cup \bigcup_{j\in[t]}R_j$ for all $i\in [2t]$.
\end{lemma}
\begin{proof}
Set $\cQ=\{Q_1,\ldots,Q_{2t}\}$ and $\cR=\{R_1,\ldots,R_{t}\}$, and
let $H'=\bigcup_{j\in[2t]}Q_j\cup \bigcup_{j\in[t]}R_j$.
Let $\cS=\{S_1,\ldots,S_t\}$ be $t$ disjoint $B$--$X$~paths in $H'$
such that 
$$\sum_{i=1}^t|E(S_i)\sm E\big( \bigcup_{Q\in\cQ}Q\big)|$$ 
is minimal.
Suppose there is a path $Q\in \cQ$ such that a path in $\cS$ intersects $Q$, 
but there is no path in $\cS$ which shares the endvertex $x\in X$ of $Q$.
Let $y$ be the first vertex from $x$ on $Q$ that belongs to a path in $\cS$; let $S\in\cS$
be the path containing $y$.
Let $S'$ be the path obtained from $S$ by deleting the subpath from $y$ to $X$ and adding $yQx$.
Then $(\cS\cup \{S'\})\sm \{S\}$ contradicts the choice of $S$.

Therefore, if $Q\in \cQ$ has a nonempty intersection with a path in $\cS$,
it also shares its endvertex in $X$ with a path in $\cS$.
Hence at most $t$ paths in $\cQ$ intersect a path in $\cS$
and so there exist $t$ paths in $\cQ$ that are disjoint from paths in $\cS$.
These paths together with $\cS$ give rise to the desired $2t$ paths.
\end{proof}

\begin{lemma}\label{pathnailslink}
Let $t\geq 2$ be a positive integer, and let {$H$} be a graph, let $A\subseteq {V(H)}$ and let
$W$ be a wall.
Let $\mathcal P$ be a set of $3t$ disjoint $A$--$W$~paths that nicely link to {$W$},
and let $\mathcal L$ be a linkage of $W$ of size $6t$.
Then there is a set $\mathcal P'$ of $t$ disjoint $A$--$W$~paths that nicely link to $W,$
and a subset $\mathcal L'$ of $\mathcal L$ of size $t$ such that the paths in $\mathcal P'\cup\mathcal L'$
are pairwise disjoint. Moreover, there is an edge $e$ in the outer cycle $C$ of $W$
such that the endvertices of the paths in $\mathcal P'$ precede the endvertices of the paths in $\mathcal L'$
in the path $C-e$.
\end{lemma}
\begin{proof}
We apply Lemma~\ref{twotypespaths} to $\mathcal P$ and $\mathcal L$ with the set of nails of $W$
in the role of $X$. We obtain a subset $\mathcal L_1$ of $\mathcal L$ of size $3t$, and a set $\mathcal P_1$
of $3t$ $A$--$W$~paths that nicely link to $W$
such that the paths in $\mathcal L_1\cup\mathcal P_1$ are pairwise disjoint.

There are three disjoint subpaths $P_1,P_2,P_3$ of $C$ such that each contains the endvertices of $t$
of the paths in $\mathcal P_1$.
For $i=1,2,3$, let $J_i$ be the set of those paths  in $\mathcal L_1$
that do not have an endvertex in $P_i$. Pick $i$ such that $J_i$ is largest, which
implies $|J_i|\geq\tfrac{1}{3}|\mathcal L_1|=t$, and set $\mathcal L'=J_i$. We also choose
as $\mathcal P'$ the set of those paths in $\mathcal P_1$ with an endvertex in $P_i$.
Clearly, $|\mathcal P'|=t$.
To finish the proof, it remains to pick an edge $e$ of $C$ and an orientation of $C$ such that $C-e$
has $P_i$ as initial segment.
\end{proof} 

We also need a result of Thomassen that any large enough wall contains a large wall 
in which all subdivided edges have a length that is divisible by~$m$, for any fixed positive integer $m$.  

\begin{proposition}[Thomassen~\cite{Tho88}]\label{thoprop}
For all positive integers $m,\ell$ there is an integer $\thofun(\ell,m)$ such that every $\thofun(\ell,m)$-wall 
contains an $\ell$-wall such that all subdivided edges have length~$0$ modulo~$m$. 
\end{proposition}

\begin{lemma}\label{pathlinkagetotop}
Let $k$ be some positive integer and $s=200k$.
Let $\plfun(k)$ be an integer with $\plfun(k) \geq \thofun(4s\cdot 2 \cdot 3s^2,4)+400k+2$.
If $G$ contains a bipartite  wall $W_0$
of size $\plfun(k)$ that has a  pure odd linkage $\mathcal L$ of size $48(k+1)$,
and if there is a set  $\mathcal P$ of $24(k+1)$ disjoint $A$--$W_0$~paths that nicely link to $W_0$
then there are $k$ disjoint zero $A$-paths.
\end{lemma}

\begin{proof}
Let $W_1$ be a $200k$-contained subwall of $W_0$ of size $\plfun(k)-400k$. Let $A_1$ be the set of
nails of $W_1$ and let $W_2$ be the $1$-contained subwall of $W_1$ of size $\plfun(k)-400k-2$.
By Proposition~\ref{thoprop}
and choice of $\plfun(k)$,
$W_2$ contains a wall $W_3$ of size $4s\cdot 2\cdot 3s^2$ such that all subdivided edges of $W_3$ 
have length~$0$ modulo~$4$. Note that $A_1$ is disjoint from $W_2$ and therefore also $W_3$.
We use Lemma~\ref{linklem} on $A_1$ and $W_3$ and thus find a subwall $W_4$ of $W_3$ of size $2\cdot 3s^2$ such that
there is either a set of $s$ disjoint paths from $A_1$ to the nails of $W_4$ or a set $X$ of
fewer than $3s^2$ vertices that separate $A_1$ from $W_4$. 

Assume there is such a separator $X$. 
Each branch vertex of $W_4$, with the possible exception of the vertices of degree $2$, is a branch vertex in $W_2$. 
Since $W_4$ has size $2\cdot 3s^2$ and since $|X|<3s^2$,
 we find a horizontal or a vertical path $P$ of $W_1$ that is 
disjoint from $X$ such that at least one branch vertex of $W_4$ is contained in this path.
Since the number of vertices in $A_1$ is larger than $3s^2$,
 we also find a path starting in a vertex of $A_1$
and ending at the bottom of the wall $W_1$ (an extension of a vertical path of $W_1$) that is disjoint from $X$. Either it intersects the path $P$ or we find a horizontal path of $W_1$
that intersects both these paths and is also disjoint from $X$. In any case we obtain a path from $A_1$ to a branch vertex of $W_4$ that is disjoint from $X$ --- this contradicts that $X$ separates $A_1$ from $W_4$.
Thus, Lemma~\ref{linklem} yields a set $\mathcal R$ of $s$ disjoint 
paths starting in $A'_1 \subseteq A_1$ that nicely link to $W_4$.
Note that, by applying the lemma in $W_1$, we can ensure that each path in 
$\mathcal R$ is contained in~$W_1$.

Recall that the linkage $\mathcal L$ has size $48(k+1)$, and that thus 
the set of endvertices of the paths in $\mathcal L$ has cardinality $2\cdot 48(k+1)\leq 200k=s$.
As $W_1$ is $200k$-contained in $W_0$
we can extend the pure linkage $\mathcal L$
through $W_0$ to a pure linkage of $W_1$ such that all endvertices are in $A'_1$.
We  now use the $A'_1$--$W_4$~paths in $\mathcal R$ to extend the linkage of $W_1$ to one
of $W_4$. We denote the linkage by~$\mathcal L_1$, and observe that $\mathcal L_1$
is still an odd linkage as $W_0$ is bipartite and as every path we used 
to extend the paths in $\mathcal L$ is contained in $W_0$. 
 
Moreover, as all branch vertices of $W_4$ are in the same bipartition class, we deduce
that every path in the linkage {$\mathcal L_1$} has odd length. 
In a similar way, we extend $\mathcal P$ first through $W_0$ and then via the paths in $\mathcal R$
to a  set $\mathcal P_1$ of $24(k+1)$ disjoint $A$--$W_4$~paths
that nicely link to $W_4$.

Let $Q$ denote the top row of $W_4$. 
Next, we apply Lemma~\ref{pathnailslink} to $\mathcal P_1$ and $\mathcal L_1$ with $t=8(k+1)$.
We obtain a set $\mathcal P_2$ of $8(k+1)$ disjoint $A$--$W_4$~paths that nicely link to $W_4$ and
a subset $\mathcal L_2\subseteq\mathcal L_1$ of size $4(k+1)$ such that the paths in $\mathcal P_2\cup\mathcal L_2$
are pairwise disjoint. Moreover, there is a subpath of $Q$ that contains all endvertices of the
paths in $\mathcal P_2$ but no endvertex of any path in $\mathcal L_2$. 

Now, as the paths in $\mathcal L_2$ have odd length, there are at least $2(k+1)$
paths in $\mathcal L_2$ that have the same length modulo~$4$, namely either~$1$ or~$3$. 
Let these linkage paths be $\mathcal L_3=\{L_1,\ldots, L_{2(k+1)}\}$, where we assume the paths to be ordered
according to the order of their first endvertex on the top row $Q$.
If $\mathcal L_3$ is crossing or nested, then for $i=1,\ldots, 2(k+1)$, denote the first 
endvertex of $L_i$ on $Q$ by $a_i$ and the other endvertex by $b_i$. If $\mathcal L_3$ 
is in series, then let the endvertices of $L_i$ be $a_i$ and $b_i$, where we choose them 
such that $a_i$ is the first endvertex on $Q$ if and only if $i$ is odd. 
In both cases the subpaths $b_{2j-1}Qb_{2j}$ for $j=1,\ldots, k+1$ are disjoint. 
In particular, at most one of these subpaths may contain an endvertex of any path in $\mathcal P_2$.
If that happens, for $j^*$ say, we delete the two paths $L_{2j^*-1}$ and $L_{2j^*}$ from $\mathcal L_3$.
To keep notation simple, we rename, in that case, the remaining paths in $\mathcal L_3$
so that $\mathcal L_3$ consists of $L_1,\ldots, L_{2k}$, ordered according to their endvertices in $Q$.

We get:
\begin{equation}\label{disjQ}
\begin{minipage}[c]{0.8\textwidth}\em
for $j=1,\ldots, k$,  the only endvertices of any path in $\mathcal L_3\cup \mathcal P_2$
contained in $b_{2j-1}Qb_{2j}$ are $b_{2j-1}$ and $b_{2j}$.
\end{minipage}\ignorespacesafterend 
\end{equation} 

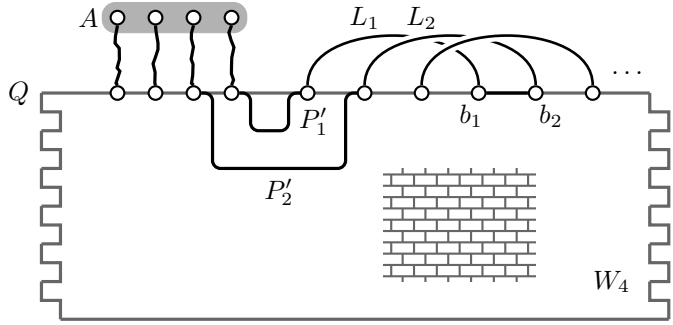
\begin{figure}[ht]
\centering
\begin{tikzpicture}
\tikzstyle{hvertex}=[thick,circle,inner sep=0.cm, minimum size=1.8mm, fill=white, draw=black]
\tikzstyle{wed}=[hedge,color=dunkelgrau]
\tikzstyle{jumps}=[ultra thick, white, double distance=1pt, double=black,bend left=90]

\def\bh{0.25}
\def\downhook{-- ++(0,-\bh) -- ++(\bh,0) -- ++(0,-\bh) -- ++(-\bh,0)}
\def\lasthook{-- ++(0,-\bh) -- ++(\bh,0) -- ++(0,-\bh)}
\def\wallh{6}
\def\rightx{8}

\begin{scope}[on background layer]

\foreach \j in {2,...,\wallh}{
  \draw[wed] (0,-\j*2*\bh+4*\bh) \downhook;
  \draw[wed] (\rightx,-\j*2*\bh+4*\bh) \downhook;
}
\draw[wed] (0,-\wallh*2*\bh+2*\bh) \lasthook;
\draw[wed] (\rightx,-\wallh*2*\bh+2*\bh) \lasthook;

\draw[wed] (0,0) -- (\rightx,0);
\draw[wed] (\bh,-\wallh*2*\bh) -- (\rightx+\bh,-\wallh*2*\bh);
\end{scope}

\foreach \i in {0,1,2,3,4,5}{
  \node[hvertex] (z\i) at (3.5+\i*3*\bh,0) {};
}

\draw[jumps] (z0) to (z3);
\draw[jumps] (z1) to (z4);
\draw[jumps] (z2) to (z5);

\draw[color=hellgrau,fill=hellgrau,rounded corners=5pt] (0.8,0.8) rectangle (1.2+6*\bh,1.2);	

\foreach \i in {0,1,2,3}{
  \node[hvertex] (a\i) at (1+\i*2*\bh,1) {};
  \node[hvertex] (b\i) at (1+\i*2*\bh,0) {};
  \draw[pathedge] (a\i) to (b\i);
}

\draw[hedge] (z3) -- (z4);

\node[label=above right:$\ldots$] at (z5){};
\node[label=left:$A$] at (a0){};
\node at (\rightx-0.5,-\wallh*2*\bh+0.5) {$W_4$};
\node at (3.5+3*\bh,1) {$L_1$};
\node at (3.5+6*\bh,1) {$L_2$};
\node at (-0.3,0) {$Q$};

\path (z3) ++(-0.1,-0.3) node {$b_1$};
\path (z4) ++(0.2,-0.3) node {$b_2$};

\begin{scope}[on background layer]
  \draw[hedge, rounded corners=3pt] (b3) -- ++(\bh,0) -- ++(0,-0.5) -- ++(2.5-8*\bh,0) -- ++ (0,0.5) node[near start, auto,swap] {$P'_1$} -- ++(\bh,0);
  \draw[hedge, rounded corners=3pt] (b2) -- ++(\bh,0) -- ++(0,-1) -- ++(2.5-3*\bh,0) node[midway,auto,swap]{$P'_2$} -- ++ (0,1) -- ++(\bh,0);
\end{scope}

\clip[] (4.5,-2.5) rectangle (6.5,-1);

\begin{scope}[shift={(4,-2.58)}]
\tikzstyle{wed}=[thick,color=dunkelgrau]
\def\wallheight{12}
\def\brickheight{0.15}

\pgfmathtruncatemacro{\lastrow}{\wallheight}
\pgfmathtruncatemacro{\penultimaterow}{\wallheight-1}
\pgfmathtruncatemacro{\lastrowshift}{mod(\wallheight,2)}
\pgfmathtruncatemacro{\lastx}{2*\wallheight+1}

\draw[wed] (\brickheight,0) -- (2*\wallheight*\brickheight+\brickheight,0);
\foreach \i in {1,...,\penultimaterow}{
  \draw[wed] (0,\i*\brickheight) -- (2*\wallheight*\brickheight+\brickheight,\i*\brickheight);
}
\draw[wed] (\lastrowshift*\brickheight,\lastrow*\brickheight) to ++(2*\wallheight*\brickheight,0);

\foreach \j in {0,2,...,\penultimaterow}{
  \foreach \i in {0,...,\wallheight}{
    \draw[wed] (2*\i*\brickheight+\brickheight,\j*\brickheight) to ++(0,\brickheight);
  }
}
\foreach \j in {1,3,...,\penultimaterow}{
  \foreach \i in {0,...,\wallheight}{
    \draw[wed] (2*\i*\brickheight,\j*\brickheight) to ++(0,\brickheight);
  }
}
\end{scope}

\end{tikzpicture}
\caption{How the zero $A$-paths are pieced together in the proof of Lemma~\ref{pathlinkagetotop}}\label{castlefig}
\end{figure}

Of the paths in $\mathcal P_2$ at least $2k$ have the same length modulo~$4$.
Let the set of these be $\mathcal P_3=\{P_{1},\ldots, P_{2k}\}$, and assume them to 
be ordered according to their endvertices $p_{1},\ldots,p_{2k}$ on $Q$. 
If the paths in $\mathcal P_3$
have length~$0$ or length~$2$ modulo~$4$ then for $j=1,\ldots, k$ the 
path $Q_j=P_{2j-1}p_{2j-1}Qp_{2j}P_{2j}$ has length~$0$ modulo~$4$: indeed, as 
$p_{2j-1}$ and $p_{2j}$ are nails of $W_4$ it follows that $p_{2j-1}Qp_{2j}$ has length~$0$
modulo~$4$. As, moreover, the paths $Q_1,\ldots, Q_k$ are all pairwise disjoint, we 
have found $k$ disjoint zero $A$-paths.

Thus, assume the paths in $\mathcal P_3$ to have length~$1$ or~$3$ modulo~$4$.
Following the outer cycle in the right direction (right with respect to the top row) let $a_\ell$ be the first vertex after $p_{2k}$. We start by relabelling this vertex $a_\ell$ (and the respective linkage path) as $a_{2k}$ and then relabel all following vertices $a_i$ with decreasing index such that we get $a_{2k}, \ldots, a_1$ if we follow the outer cycle.

We extend each path $P_i$ in $\mathcal P_3$ through $W_4$ to an $A$--$a_i$~path $P'_i$
and note that, as $W_4$ has size $10k$,
 we can do that in such a way that $P'_1,\ldots,P'_{2k}$ are pairwise disjoint,
such that they meet none of the subpaths $b_{2j-1}Qb_{2j}$ for $j\in\{1,\ldots,k\}$, 
and such that they also meet the paths of $\mathcal L_3$ only in their endvertices. 
By choice of $W_2$ the paths $P'_i$ still all have the same length modulo~$4$, namely~$1$
or~$3$. 
Define $Q_j$ as 
\[
Q_j=P'_{2j-1}a_{2j-1}L_{2j-1}b_{2j-1}Qb_{2j}L_{2j}a_{2j}P'_{2j},
\]
and observe that $Q_j$ is a zero $A$-path  as the sum of the lengths of $L_{2j-1}$ and $L_{2j}$
is~$2$ modulo~$4$. See Figure~\ref{castlefig} for an illustration. 	
By choice of the $P'_i$ and and by~\eqref{disjQ},
the paths $Q_j$ are pairwise disjoint, and thus disjoint zero $A$-paths.
\end{proof}

\begin{lemma} \label{linkagewallminor}
There is an integer $\bfun(k)$ such that: if $W_0$ is a bipartite  wall in $G-A$
of size at least $\bfun(k)$ whose induced tangle is a truncation of \eptan, and 
that has a pure odd linkage $\mathcal L$ of size at least $\frac{\bfun(k)}{100}$ then 
$G$ contains  $k$ disjoint zero $A$-paths.
\end{lemma}

\begin{proof}
Set $t=24(k+1)$, and choose $r$ such that $r\geq 3t^2$ and $r\geq \plfun(k)$.
Next, define $\bfun(k)$ such that $\bfun(k)\geq 100\cdot 48(k+1)$ and 
$\bfun(k)\geq 4tr+400k$. Note that $\bfun$ only depends on $k$. 

Choose $W_1$ to be  a $200k$-contained subwall of $W_0$ of size $4tr$.
We first apply Lemma~\ref{linklem2}, and obtain a subwall $W_2$ of $W_1$ 
of size at least $r$ and a set of $t=24(k+1)$ disjoint $A$--$W_2$~paths
that nicely link to $W_2$ (here we use that the tangle induced by $W_0$
is a truncation of \eptan).
 
We extend $48(k+1) \leq \frac{\bfun(k)}{100}$ of the 
linkage paths in $\mathcal L$ through $W_0$ to a linkage $\mathcal L_0$
of $W_2$ of the same size (this is possible since $W_2$ is a subwall of $W_1$ that is $200k$-contained in $W_0$). 
As $W_0$ is bipartite, $\mathcal L_0$ is still odd. 
We conclude the proof by applying Lemma~\ref{pathlinkagetotop} to $W_2$. 
\end{proof}

Now we can finally finish with the proof of Theorem~\ref{zeroepp}.
Recall that the tangle \eptan\ of $G-A$ that we get from  Lemma~\ref{largetanglem}
has order $g(k)$. In particular, 
it follows from Theorem~\ref{robseywalltangle} and~\eqref{choiceg}
that $G-A$ contains a wall  $W\subseteq G-A$ 
of size $\fptfun(t^*(k))$ whose induced tangle is a truncation of $\eptan$. 
We apply Theorem~\ref{oddflatwallthm} to the wall~$W$. 

In the first three outcomes of Theorem~\ref{oddflatwallthm} we get $k$ disjoint zero $A$-paths:
for outcome~(a), this is proved in Lemma~\ref{cliqueminor} --- note that $t^*(k)\geq 16k$ by~\eqref{choicet}; 
for 
outcome~(b.i) this is proved in Lemma~\ref{nonzerowallminor} --- note that $t^*(k)\geq 2600k^3$ by~\eqref{choicet}; 
and for~(b.ii)
this is done in Lemma~\ref{linkagewallminor} --- note that $t^*(k)\geq\bfun(k)$ by~\eqref{choicet}.

It remains to treat outcome~(c), i.e., when there is a vertex set $Z$ of size 
$|Z| \leq \fptfun(k)$ such that the $\mathcal T_W$-large block $B$ of $G-A$ is bipartite.
If {there is no zero $A$-path such that its interior} is contained in a $B$-bridge in $G-A-Z$ then 
 Lemma~\ref{bipblockcase} finishes the proof, where we note that the hitting set size
there is bounded by $\pfun(k)\leq f(k)$, by~\eqref{choicef}. 

So, suppose there is some zero $A$-path $P$ in $G-Z$ that meets $B$ at most in a cutvertex $x$
of $G-A-Z$.
Then $Z\cup\{x\}$ is a set of size at most $\fptfun(k)+1<g(k)$, by~\eqref{choiceg}, 
	that separates $P$ from the $\mathcal T_W$-large block $B$ in $G-Z$.
As $\mathcal T_W$ is a truncation of \eptan, we obtain a contradiction to Lemma~\ref{largetanglem}. 
This concludes the proof of Theorem~\ref{zeroepp}.

\section{Conclusion}

We have proved that $A$-paths with a length congruent to~$0$ modulo~$4$ have the Erd\H os-P\'osa
property. What happens when we fix $d\in\{1,2,3\}$ and consider $A$-paths of length congruent to~$d$ 
modulo~$4$ instead? The answer for $d=2$ is an easy consequence of Theorem~\ref{zeroepp}.

\begin{proposition}
$A$-paths of length $2$ modulo $4$ have the Erd\H{o}s-P\'osa property.
\end{proposition}
\begin{proof}
Given a graph $G$ with a vertex set $A\subseteq V(G)$, we first observe that we may assume 
that $G$ has no edge with both endvertices in $A$. Indeed, any such edge is an $A$-path of length~$1$
and not contained in any $A$-path of any other length. 
Next, we subdivide every edge incident with a vertex in $A$ once. Call the resulting graph $G'$. 
Then, each $A$-path of length~$2$ modulo~$4$ in $G'$ corresponds to an 
$A$-path of length $0$ modulo $4$ in $G$, and vice versa. Applying Theorem~\ref{zeroepp} to $G'$
finishes  the proof.
\end{proof}


For $d=1$ or $d=3$, on the other hand, the Erd\H{o}s-P\'osa property is not satisfied. 
This can be seen by a construction that is very similar to one developed 
for $A$-paths of length~$0$ modulo~$m$, for a non-prime $m\geq 6$; see~\cite{BHJ18b}.

\begin{proposition}
For $d \in \{1,3\}$,
$A$-paths of length $d$ modulo $4$ do not have the Erd\H{o}s-P\'osa property.
\end{proposition}

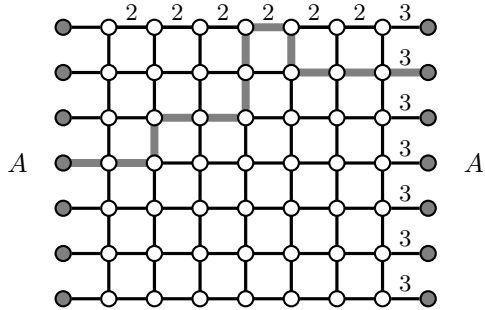
\begin{figure}[htb]
\centering
\begin{tikzpicture}
\tikzstyle{Apath}=[line width=3pt,grau];

\def\hstep{0.6}
\def\height{7}

\begin{scope}[on background layer]
\foreach \i in {1,...,\height}{
  \draw[hedge] (\hstep,\i*\hstep) -- (\height*\hstep,\i*\hstep);
  \draw[hedge] (\i*\hstep,\hstep) -- (\i*\hstep,\height*\hstep);
}
\end{scope}

\foreach \i in {2,...,\height}{
  \begin{scope}[on background layer]
    \draw[hedge] (\i*\hstep-\hstep,\height*\hstep) to coordinate[midway] (tl\i) (\i*\hstep,\height*\hstep);
  \end{scope}
  \path (tl\i) ++(0,0.2) node{{\small $2$}};
}

\foreach \i in {1,...,\height}{
  \foreach \j in {1,...,\height}{
    \node[hvertex] (t\i\j) at (\i*\hstep,\j*\hstep){};
  }
}

\foreach \j in {1,...,\height}{
  \node[hvertex, fill=grau] (la\j) at (0,\j*\hstep){};
  \node[hvertex, fill=grau] (ra\j) at (\height*\hstep+\hstep,\j*\hstep){};
  \begin{scope}[on background layer]
    \draw[hedge] (la\j) to (t1\j);
    \draw[hedge] (ra\j) to coordinate[midway] (rl\j) (t\height\j);
  \end{scope}
  \path (rl\j) ++(0,0.2) node{{\small $3$}};
}

\begin{scope}[on background layer]
\draw[Apath] (la4.center) -- (t14.center) -- (t24.center) -- (t25.center) -- (t35.center) -- (t45.center) -- (t46.center) -- (t47.center) -- (t57.center) -- (t56.center) -- (t66.center) -- (t76.center) -- (ra6.center);
\end{scope}

\node at (-\hstep,0.5*\height*\hstep+0.5*\hstep) {$A$};
\node at (\height*\hstep+2*\hstep,0.5*\height*\hstep+0.5*\hstep) {$A$};

\end{tikzpicture}
\caption{All unlabeled edges have length~$4$; an $A$-path of length $\equiv 1$ (mod~$4$) in grey.}\label{zeroAfig}
\end{figure}

\begin{proof}
Suppose that every graph either contains two disjoint $A$-paths of length~$d$ modulo~$4$,
or  a set of at most $f(2)$ many vertices that meet every such $A$-path.

Consider a grid of size $10f(2)$, and subdivide every edge in the grid, except for those in the top row, 
three times, such that they become paths of length~$4$. Subdivide the edges in the top row once, 
so that they turn into paths of length~$2$.
Add a set $A$ of $20f(2)$ new vertices, 
pick half of the  vertices in $A$, and connect each in the half 
 to a distinct branch vertex  
on the left boundary of the (subdivided) grid, via pairwise disjoint paths of length~$4$. 
We connect the other half of $A$ in the same way to the branch vertices on the right boundary of 
the (subdivided) grid, only we use paths of length $d+2$ instead of~$4$; 
see Figure~\ref{zeroAfig}.

Any $A$-path that starts and ends on the left, or starts and ends on the right,
 has even length, and in particular not length~$d$ modulo~$4$. Any $A$-path that starts
on the left and ends on the right but is disjoint from the top row has length~$d+2$ modulo~$4$.
Thus, the only $A$-paths of length~$d$ modulo~$4$ are those that start on the left, 
traverse at least one edge in the top row and then end on the right. Clearly, there cannot
be two disjoint such paths. 

Thus, by assumption, there should be a set of at most $f(2)$ 
vertices that meets every $A$-path of length~$d$ modulo~$4$. 
This, however, is easily seen to be false. 
 Therefore, the paths of length $d$ modulo $4$ do not have the Erd\H{o}s-P\'osa property.
\end{proof}

What about prime $m$?
\begin{problem}
For $m>2$ prime, 
do $A$-paths of length~$0$ modulo~$m$ have the Erd\H{o}s-P\'osa property?
\end{problem}
We suspect that the answer is ``yes''. Unfortunately, though, the methods
we use cannot easily be adapted to the prime case. The reason for this lies in our main 
tool, Theorem~\ref{oddflatwallthm} of Huynh, Joos and Wollan. 
Clearly, our simplified version of the theorem is useless for prime $m>2$
but also the original version will probably not help. This is because the original, stronger
version assumes a group labelling on the edges of an \emph{orientation} of the graph.
That is, if an edge $e$ is traversed in one direction we will pick up a 
group element~$\alpha$ (perhaps~$1$), but if $e$ is traversed in the opposite
direction then we pick up $-\alpha$. 
This feature makes it difficult to work out whether a certain path has length~$0$ 
modulo~$m$, as the length is inherently undirected: the length stays the same 
in whatever direction we traverse the edges. 

\subsection*{Acknowledgment}

We thank Felix Joos for fruitful discussions and the donation of Lemmas~\ref{twotypespaths} and~\ref{pathnailslink}.

\bibliographystyle{amsplain}
\bibliography{erdosposa}

\vfill

\small
\vskip2mm plus 1fill
\noindent
Version \today{}
\bigbreak

\noindent
Henning Bruhn
{\tt <henning.bruhn@uni-ulm.de>}\\
Arthur Ulmer
{\tt <arthur.ulmer@uni-ulm.de>}\\
Institut f\"ur Optimierung und Operations Research\\
Universit\"at Ulm\\
Germany\\

\end{document}